\documentclass[10pt,a4paper]{article}
\usepackage[utf8]{inputenc}
\usepackage{a4wide}
\usepackage{amsmath}
\usepackage{amsfonts}
\usepackage{amssymb}
\usepackage{graphicx}
\usepackage{hyperref}
\usepackage{enumitem}
\usepackage{color}

\author{Benne de Weger \\ {\small Eindhoven University of Technology, Eindhoven, The Netherlands} \\
    {\small\url{b.m.m.d.weger@tue.nl}}}
\title{On the approximation gain for $ abc $-triples}
\newcommand{\version}{1.3 (arXiv version 2)}
\date{version \version, \today}

\newcommand{\Z}{\mathbb{Z}}

\newcommand{\dsum}{\displaystyle\sum}
\DeclareMathOperator{\rad}{rad}
\DeclareMathOperator{\qu}{qu}
\DeclareMathOperator{\rrad}{rrad}
\DeclareMathOperator{\rag}{rag}
\DeclareMathOperator{\rpg}{rpg}
\DeclareMathOperator{\mrad}{mrad}
\DeclareMathOperator{\mpag}{mag}
\DeclareMathOperator{\mpg}{mpg}
\DeclareMathOperator{\srad}{srad}
\DeclareMathOperator{\sag}{sag}
\DeclareMathOperator{\spg}{spg}
\DeclareMathOperator{\cag}{cag}
\DeclareMathOperator{\cpg}{cpg}

\newtheorem{definition}{Definition}
\newtheorem{theorem}{Theorem}

\reversemarginpar

\begin{document}

\maketitle

\begin{abstract}
The concept of \emph{approximation gain} was introduced recently by M\"uller and Taktikos for some
$ abc $-triples related to convergents of surds, where there is a relatively large gap between
$ \min\{a, b, c\} $ and $ \max\{a, b, c\} $. This note proposes a generalization of the concept to all
$ abc $-triples, with several variants. Extensive numerical computations are provided.
\end{abstract}

\section{Introduction}

We start with common definitions.

\begin{definition} \label{def:rad}
The \emph{radical} $ \rad(n) $ of a positive integer $ n $ is defined as the product of its prime divisors.
\end{definition}

\begin{definition} \label{def:abc}
Let $ a, b, c $ be positive coprime integers such that $ a + b = c $.
\begin{enumerate}[topsep=0pt]
\item The \emph{quality} of the triple $ (a,b,c) $ is defined as
      \[ \qu(a,b,c) = \dfrac{\log c}{\log \rad(abc)} . \]
\item A triple $ (a,b,c) $ is called an \emph{$ abc $-triple} when $ \qu(a,b,c) > 1 $.
\end{enumerate}
\end{definition}

The $ abc $-conjecture says that for every $ \epsilon > 0 $ there are only finitely many $ abc $-triples
with $ \qu(a,b,c) > 1 + \epsilon $.

The concept of \emph{approximation gain} for $ abc $-triples was introduced recently by M\"uller and
Taktikos \cite{MT}. Their starting point is the surd $ k^{1/s} $ where $ k, s $ are positive integers, with
convergents $ \dfrac{p}{q} $, so that $ d = p^s - k q^s $ is relatively close to $ 0 $. This gives a
triple of quality at least close to $ 1 $. To be a bit more precise, when $ e $ is the partial
quotient occurring just after the convergent $ \dfrac{p}{q} $, then by continued fraction theory
$ \left| k^{1/s} - \dfrac{p}{q} \right| \approx \dfrac{1}{e q^2} $, which implies
\[ |d| = \left| p^s - k q^s \right| = \left| p - k^{1/s} q \right| \dsum_{i=0}^{s-1} p^i \left(k^{1/s}
         q\right)^{s-1-i} \approx \dfrac{s k^{1-1/s}}{e} q^{s-2} . \]
For the proper permutation $ a, b, c $ of $ d, p^s, k q^s $ this gives for the $ abc $-quality in the worst
case (where $ d k p q $ is squarefree)
\[ \qu(a,b,c) = \dfrac{\log\max\{|d|, p^s, k q^s\}}{\log(|d| p k q)} \approx
   \dfrac{\log(k q^s)}{\log\left(\dfrac{s k^2}{e} q^s\right)} \approx
   \dfrac{\log(k q^s)}{\log(k q^s) + \log(k s) - \log e} . \]
The intuition is that the partial quotients have a tendency to be small often, and almost certainly
$ e \ll q $ for growing $ q $. This means that indeed the quality will tend towards $ 1 $. When $ e $ is big
enough it comes above $ 1 $. For some $ k $ and $ s $ this actually does happen, but the $ abc $-conjecture
implies that big partial quotients $ e $ are rare and can't be extremely big. To quantify this a bit more
elegantly, M\"uller and Taktikos split up the $ abc $-quality into a product of two terms: the
\emph{approximation gain} $ \dfrac{\log\max\{p^s,k q^s\}}{\log(|d| p k q)} $ measuring the effect on the
$ abc $-quality of how well $ \dfrac{p}{q} $ approximates $ k^{1/s} $, and the cofactor called \emph{power
gain} $ \dfrac{\log(|d| p k q)}{\log\rad(|d|pkq)} $ measuring the other effects. Where the $ abc $-conjecture
has been inspired by the Mason-Stothers Theorem on the polynomial case of $ a + b = c $, M\"uller \cite{M}
shows that there are similar theorems for the decomposition of the quality into the two gains in the
polynomial case.

For the remainder of this note we assume without loss of generality that $ a < b $, then
$ a = \min\{a, b, c\} $, $ c = \max\{a, b, c\} $. Clearly in the situation of M\"uller and Taktikos one
always has $ a \ll b $. There are however many $ abc $-triples where this gap does not occur, i.e., $ a $
and $ b $ are of approximately the same size, and there are also many $ abc $-triples that have a gap but
do not immediately relate to surds. This note proposes several variants of the approximation gain concept,
starting with a \emph{real approximation gain} in an attempt to cover the cases with a gap but not
necessarily directly related to a surd, then extending this to \emph{$ p $-adic approximation gains} where
the gap is there in some $ p $-adic sense, and resulting in the combined concepts \emph{multiple $ p $-adic
approximation gain} and \emph{combined approximation gain} which are applicable to any $ abc $-triple.
For the known $ abc $-triples we present data on the accompanying web page \cite{dW26} and in tables and
statistics below.

\section{The real approximation gain}

\subsection{Introduction}

An old idea\footnote{Basically this is how Eric Reyssat found his record $ abc $-triple back in 1987.
Although the idea of using continued fractions for surds is not explicitly present in \cite{dW87}
and \cite[Chapter 6]{dW88}, the method is at least related, in using a reduction algorithm for
linear forms that sort of generalizes continued fractions, and in working well for surds too.} to find
$ abc $-triples of high quality, which works when $ a \ll b $, is to look for $ d, b', c' \in \Z $, with
$ b = \beta^d b' $, $ c = \gamma^d c' $, that allow an exceptionally good rational approximation
$ \dfrac{\gamma}{\beta} $ to $ \left(\dfrac{b'}{c'}\right)^{1/d} $, which can be easily found using continued
fractions. Then $ a = c - b $ is pretty small, and this gives hope for a good quality.

\subsection{Definitions}

We start with definitions capturing the approximation gain for surds of various degrees $ d $. As here we
only use the real metric while further on we switch to $ p $-adic metrics, we here use the term \emph{real
approximation gain}.

\begin{definition} \label{def:gain_d}
Let $ (a,b,c) $ be an $ abc $-triple, and let $ d \geq 2 $ be any integer.
\begin{enumerate}[topsep=0pt]
\item Let $ \beta $, $ b' $, $ \gamma $, $ c' $ be positive integers with $ \beta $, $ \gamma $ as large as
      possible such that $ b = \beta^d b' $, $ c = \gamma^d c' $. The \emph{real enhanced radicals}
      $ \rrad_d(a,b,c) $ are defined by
      \[ \rrad_d(a,b,c) = a \beta b' \gamma c'. \]
\item The \emph{$ d $th degree real approximation gain} $ \rag_d(a,b,c) $ and the \emph{$ d $th degree real
      power gain} $ \rpg_d(a,b,c) $ are defined by
      \[ \rag_d(a,b,c) = \dfrac{\log c}{\log\rrad_d(a,b,c)}, \quad
         \rpg_d(a,b,c) = \dfrac{\log \rrad_d(a,b,c)}{\log \rad(abc)}. \]
\end{enumerate}
\end{definition}

Note that a real enhanced radical is not necessarily squarefree, e.g.\ $ \rrad_2(1,8,9) = 12 $. Clearly
$ \qu(a,b,c) = \rag_d(a,b,c) \rpg_d(a,b,c) $, and $ \rrad_d(a,b,c) \geq \rad(abc) $ so
$ \rag_d(a,b,c) \leq \qu(a,b,c) $ and $ \rpg_d(a,b,c) \geq 1 $.

For maximizing the approximation gain it makes sense to define the \emph{real approximation gain}
$ \rag(a,b,c) $ and the \emph{real power gain} $ \rpg(a,b,c) $ as the largest resp.\ smallest possible
$ d $'th degree ones.

\begin{definition} \label{def:gain}
For an $ abc $-triple $ (a,b,c) $, the \emph{real approximation gain} $ \rag(a,b,c) $ and the \emph{real
power gain} $ \rpg(a,b,c) $ are defined by
\[ \rag(a,b,c) = \max_{d \geq 2} \rag_d(a,b,c), \quad \rpg(a,b,c) = \min_{d \geq 2} \rpg_d(a,b,c). \]
\end{definition}

It's not hard to see that they exist: when $ d $ gets bigger than the highest exponent occurring in $ b $ or
$ c $ then $ \rrad_d(a,b,c) = abc $, so only finitely many real enhanced radicals are possible. And of
course $ \qu(a,b,c) = \rag(a,b,c) \rpg(a,b,c) $, $ \rag(a,b,c) \leq \qu(a,b,c) $ and $ \rpg(a,b,c) \geq 1 $.
Also note that $ \rrad_d(a,b,c) \leq a b c < c^3 $, so $ \rag(a,b,c) > \dfrac{1}{3} $, and
$ \rpg(a,b,c) < 3 \qu(a,b,c) $.

\subsection{Examples}
As an example we take Reyssat's $ abc $-triple $ 2 + 3^{10} \cdot 109 = 23^5 $, of record quality
$ 1.62991\ldots $, with $ d = 5 $. This gives $ b' = 109 $, $ c' = 1 $, $ \beta = 3^2 $, $ \gamma = 23 $. We
compute the continued fraction $ 109^{1/5} = [2, 1, 1, 4, 77\,733, 2, \ldots] $, with convergents
$ \dfrac{2}{1} $, $ \dfrac{3}{1} $, $ \dfrac{5}{2} $, $ \dfrac{23}{9} $, $ \dfrac{1787864}{699599}, \ldots $,
and the convergent $ \dfrac{23}{9} $ occurring just before the exceptionally large partial quotient $ 77733 $
gives us exactly $ \beta $ and $ \gamma $.

Data for the real approximation and power gains for Reyssat's $ abc $-triple can be found in Table
\ref{tab:Reyssat_real}.

\begin{table}[ht] \centering
{\small $ \begin{array}{c|cc|cc|c|cc}
          d & \beta &               b' & \gamma &   c' &                     \rrad_d(a,b,c) & \rag_d(a,b,c)
          &  \rpg_d(a,b,c) \\ \hline
          2 &   3^5 &              109 &   23^2 & 23   & 2 \cdot 3^5    \cdot23^3 \cdot 109 & 0.77289\ldots
          & 2.10883 \ldots \\
          3 &   3^3 & 3      \cdot 109 &   23   & 23^2 & 2 \cdot 3^4    \cdot23^3 \cdot 109 & 0.81715\ldots
          & 1.99461 \ldots \\
          4 &   3^2 & 3^2    \cdot 109 &   23   & 23   & 2 \cdot 3^4    \cdot23^2 \cdot 109 & 0.97679\ldots
          & 1.66863 \ldots \\
          5 &   3^2 &              109 &   23   &  1   & 2 \cdot 3^2    \cdot23   \cdot 109 & 1.46283\ldots
          & 1.11421 \ldots \\
          6 &   3   & 3^4    \cdot 109 &   1    & 23^5 & 2 \cdot 3^5    \cdot23^5 \cdot 109 & 0.59037\ldots
          & 2.76079 \ldots \\
          7 &   3   & 3^3    \cdot 109 &   1    & 23^5 & 2 \cdot 3^4    \cdot23^5 \cdot 109 & 0.61585\ldots
          & 2.64658 \ldots \\
          8 &   3   & 3^2    \cdot 109 &   1    & 23^5 & 2 \cdot 3^3    \cdot23^5 \cdot 109 & 0.64363\ldots
          & 2.53236 \ldots \\
          9 &   3   & 3      \cdot 109 &   1    & 23^5 & 2 \cdot 3^2    \cdot23^5 \cdot 109 & 0.67403\ldots
          & 2.41814 \ldots \\
         10 &   3   &              109 &   1    & 23^5 & 2 \cdot 3      \cdot23^5 \cdot 109 & 0.70744\ldots
          & 2.30392 \ldots \\
    \geq 11 &   1   & 3^{10} \cdot 109 &   1    & 23^5 & 2 \cdot 3^{10} \cdot23^5 \cdot 109 & 0.48918\ldots
          & 3.33188 \ldots \\
    \end{array} $}
    \caption{$ d $'th degree real approximation and power gains for Reyssat's $ abc $-triple.}
    \label{tab:Reyssat_real}
\end{table}

Clearly at degree $ d = 5 $ the influence of the $ d $th powers is most prominent, and this is where the
convergent $ \dfrac{\gamma}{\beta} = \dfrac{23}{9} $ to $ 109^{1/5} $ shows up (note that this also happens
at $ d = 4 $). In this case $ \rag(a,b,c) = \rag_5(a,b,c) = 1.46283\ldots $.

Just for fun we also give a very bad sequence of examples. Let $ p $ be a prime such that also $ p - 6 $ and
$ p + 6 $ are prime. Then $ \rag(p-6,p+6,2p) = \qu(p-6,p+6,2p) = \dfrac{\log(2p)}{\log(2p(p^2-36))} $ and
$ \rpg(p-6,p+6,2p) = 1 $. Conjecturally there are infinitely many such triples, and if $ p \to \infty $ then
$ \rag(p-6,p+6,2p) \to \dfrac{1}{3} $, so the lower bound $ \dfrac{1}{3} $ for the approximation gain can not
be improved.

\subsection{Computational results}

The real approximation gain was computed for the $ 14482065 $ $ abc $-triples below $ 10^{18} $, and for the
known $ 9345651 $ $ abc $-triples between $ 10^{18} $ and $ 2^{63} $, from Bart de Smit's tables \cite{dS},
containing a wealth of data. The top 10 real approximation gain $ abc $-triples are given in Table
\ref{tab:top10_real}. The largest real approximation gain, $ 1.46283\ldots $, occurs, not unexpectedly, for
Reyssat's $ abc $-triple.

\begin{table}[ht] \centering
{\small $ \begin{array}{ccc|cc|c}
    a & b & c & \rag & \rpg & \qu \\ \hline
    2 & 3^{10} \cdot 109 & 23^5 & 1.46283\ldots & 1.11421\ldots & 1.62991\ldots \\
    1 & 2^4 \cdot 5 & 3^4 & 1.29203\ldots & 1 & 1.29203\ldots \\
    3 \cdot 11 & 2^7 \cdot 5^6 & 7^6 \cdot 17 & 1.28721\ldots & 1.06552\ldots & 1.37156\ldots \\
    5 & 3^{11} & 2^{10} \cdot 173 & 1.25190\ldots & 1.12842\ldots & 1.41268\ldots \\
    1 & 2^6 \cdot 3 \cdot 5 \cdot 7 \cdot 13^4 \cdot 17 & 239^4 & 1.24408\ldots & 1.08545\ldots & 1.35040\ldots \\
    1 & 2 \cdot 3^4 \cdot 5 \cdot 11^4 & 7 \cdot 13 \cdot 19^4 & 1.22892\ldots & 1 & 1.22892\ldots \\
    5 & 2 \cdot 3 \cdot 257^3 & 467^3 & 1.22139\ldots & 1 & 1.22139\ldots \\
    2^2 & 3^{10} \cdot 5 \cdot 19 \cdot 109 \cdot 67751 & 23^{10} & 1.20645\ldots & 1.02740\ldots & 1.23951\ldots \\
    2 \cdot 61   & 3^5 \cdot 7^5 & 11 \cdot 13^5 & 1.18821\ldots & 1 & 1.18821\ldots \\
    3 \cdot 2713 & 2^{44} \cdot 23 \cdot 107 & 5 \cdot 7^{14} \cdot 17 \cdot 751 & 1.15075\ldots & 1.09085\ldots & 1.25530\ldots \\
    \end{array} $}
\caption{Top 10 real approximation gains known.}
\label{tab:top10_real}
\end{table}

The largest real power gain observed is $ \rpg(a,b,c) = 2.94376\ldots $ at
$ (a,b,c) = (2^{49}, 7^7 \cdot 19^3 \cdot 123499, 3^{13} \cdot 5^5 \cdot 503^2) $, with
$ \rag(a,b,c) = 0.45020\ldots $, $ \qu(a,b,c) = 1.32528\ldots$. Note that in Table \ref{tab:top10_real}
$ \rag $ is always close to $ \qu $. Much more data are available on the website \cite{dW26} for this paper.

\section{The $ p $-adic approximation gain}

\subsection{Introduction}

A potential disadvantage of the real approximation gain is that it gives interesting results only when
$ a \ll b $, and that is not typical for $ abc $-triples. As an illustration we look at the example of
$ a = 2^{19} \cdot 13 \cdot 103 $, $ b = 7^{11} $, $ c = 3^{11} \cdot 5^3 \cdot 11^2 $, with a pretty good
quality $ 1.45261\ldots $, while $ a $ is relatively big, indeed $ a > \dfrac{1}{4} c $. For the real
approximation gain we get
\[ \rag(a,b,c) = \rag_{11}(a,b,c) = \dfrac{\log\left(3^{11} \cdot 5^3 \cdot 11^2\right)}{\log\left(2^{19}
    \cdot 13 \cdot 103 \cdot 7 \cdot 3 \cdot 5^3 \cdot 11^2\right)} = 0.65708\ldots , \]
which is quite low compared to the quality.

This leads to the following variation on the theme of approximation gains. Instead of looking at $ a $ that
is small in the real sense, we could look at any of $ a, b $ or $ c $ that is small in some other sense. Then
the $ p $-adic metric comes to mind, because `small' then means `divisible by a high power of $ p $'. Indeed,
for good $ abc $-triples one expects high prime powers in all three of $ a, b, c $, meaning that each of them
can be considered small in a $ p $-adic sense for some well chosen primes $ p $, or even a combination of
primes.

\subsection{Examples}

We start with a number of examples to get a feeling for what we should use as definitions. For a $ p $-adic
integer $ x = \dsum_{i=0}^{\infty} x_i p^i $ with $ x_i \in \{ 0, 1, \ldots, p-1 \} $ we write
$ x = 0.x_0x_1x_2\ldots $, where we omit the trailing zeros if $ x \in \Z $.

\subsubsection{Example 1 -- big gap}

First let's take the example of the previous subsection, $ a = 2^{19} \cdot 13 \cdot 103 $, $ b = 7^{11} $,
$ c = 3^{11} \cdot 5^3 \cdot 11^2 $, with the pretty good quality $ 1.45261\ldots $ but a quite disappointing
real approximation gain of $ 0.65708\ldots $, due to $ a $ being big in the real sense. However, $ a $ is
small in the $ p $-adic sense for $ p = 2 $, and $ b $ and $ c $ both have a big exponent $ 11 $ at other
primes. So the numbers $ \dfrac{3^{11}}{7^{11}} $ and $ (5^3 11^2)^{-1} $ are pretty close in the $ 2 $-adic
sense, indeed,
\begin{eqnarray*}
    \dfrac{3^{11}}{7^{11}} & = & 0.10111\,10000\,11110\,01011\,11000\,\ldots , \\
           (5^3 11^2)^{-1} & = & 0.10111\,10000\,11110\,01010\,11111\,\ldots ,
\end{eqnarray*}
agreeing in exactly the first $ 19 $ binary digits, corresponding to the exponent $ 19 $ of the prime $ 2 $
in $ a $. Turning to $ 11 $th roots, we compute, using Hensel's Lemma,
\begin{eqnarray*}
          \dfrac{3}{7} & = & 0.10100\,10010\,01001\,00100\,10010\,\ldots , \\
    (5^3 11^2)^{-1/11} & = & 0.10100\,10010\,01001\,00101\,11111\,\ldots ,
\end{eqnarray*}
again agreeing in exactly the first 19 binary digits. Instead of working with the surd itself, we can work
with an integer that is at $ 2 $-adic distance at most $ 2^{19} $, indeed, with
$ (5^3 11^2)^{-1/11} \approx 0.10100\,10010\,01001\,0010 = 149797 $. Note that, when looking for
$ abc $-triples starting from a surd, we don't know a priori at which exponent of $ 2 $ we're going to find
something, so we need to go through a whole set of exponents, which corresponds to not knowing at which point
we'll find a big partial quotient in a continued fraction expansion of a real surd. But when we start out
with an $ abc $-triple, we exactly know the right exponent, in this case $ 19 $.

The set of all rational approximations $ \dfrac{x}{y} $ to a given $ p $-adic integer $ \alpha $ of a given
approximation order $ n $, on identifying the fraction $ \dfrac{x}{y} $ with a vector
$ \begin{pmatrix} x \\ y \end{pmatrix} $, forms a twodimensional lattice
\[ \Lambda_n(\alpha) =
    \left\{ \left. \begin{pmatrix} x \\ y \end{pmatrix} \in \Z^2 \right| | x - y \alpha |_p \leq p^{-n}
    \right \} , \]
see \cite{dW86} (note that we see lattice points as column vectors). A basis of $ \Lambda_n(\alpha) $ is
given by the columns of the matrix $ \begin{pmatrix} p^n & \alpha_p \\ 0 & 1 \end{pmatrix} $, where
$ \alpha_p $ is the unique integer in $ \{ 0, 1, \ldots, p^n - 1 \} $ with
$ \alpha_p \equiv \alpha \pmod{p^n} $. Finding the best approximations therefore can be done efficiently by
lattice basis reduction. Note that twodimensional lattice basis reduction is closely related to the continued
fraction algorithm and Euclid's algorithm.

In our example we have the matrix $ \begin{pmatrix} 2^{19} & 149797 \\ 0 & 1 \end{pmatrix} $. The lattice
basis reduction algorithm reduces the longer column by the smaller one, meaning that one subtracts the
smaller column $ k $ times from the longer one where $ k $ is chosen optimally so that the resulting vector
gets the shortest length, and then repeating this until no more improvement is possible. In this case:
\[ \begin{pmatrix} 2^{19} & 149797 \\ 0 & 1 \end{pmatrix} \stackrel{k=3}{\to}
   \begin{pmatrix} 74897 & 149797 \\ -3 & 1 \end{pmatrix} \stackrel{k=2}{\to}
   \begin{pmatrix} 74897 & 3 \\ -3 & 7 \end{pmatrix} \stackrel{k=3874}{\to}
   \begin{pmatrix} 63275 & 3 \\ -27121 & 7 \end{pmatrix}. \]
Clearly the factors $ k $ play the role of the partial quotients, and we indeed recover $ \dfrac{3}{7} $ as a
very good rational approximation to the surd $ (5^3 11^2)^{-1/11} $ just before a huge $ k $.

One could also use a real continued fraction here, namely of $ \dfrac{149797}{2^{19}} $, being
$ [0, 3, 2, 24965, 1, 2] $, and the convergent found just before the big partial quotient $ 24965 $ is
$ \dfrac{2}{7} $. Its interpretation is that $ \dfrac{2}{7} $ is an excellent \emph{real} approximation to
$ \dfrac{149797}{2^{19}} $, so that $ 7 \cdot 149797 - 2 \cdot 2^{19} $ is an extremely small integer, in
fact it is $ 3 $. Taken modulo $ 2^{19} $ we find
$ 7 \cdot (5^3 11^2)^{-1/11} \equiv 7 \cdot 149797 \equiv 3 \pmod{2^{19}} $, revealing the very good
$ 2 $-adic approximation $ \dfrac{3}{7} $ to $ (5^3 11^2)^{-1/11} $. Going back to lattices, this precisely
corresponds to $ \begin{pmatrix} -2 \\ 7 \end{pmatrix} $ being the coordinate vector of the lattice point
$ \begin{pmatrix} 3 \\ 7 \end{pmatrix} $ w.r.t.\ the original basis.

How to define a $ p $-adic approximation gain to take advantage of the fact that now we see $ a $ as small
because it is so $ 2 $-adically? It seems to make sense to take the already defined real approximation gain
where in the denominator the $ a $ was present in full glory, and simply divide $ a $ by the relevant power
of $ 2 $, leaving one factor $ 2 $. So in this case it becomes
\[ \dfrac{\log\left(3^{11} \cdot 5^3 \cdot 11^2\right)}{\log\left(2 \cdot 13 \cdot 103 \cdot 7 \cdot 3 \cdot
    5^3 \cdot 11^2 \right)} = 1.05580\ldots , \]
indeed a much better approximation gain than the real one.

\subsubsection{Example 2 -- multiple primes} \label{sec:3.2.2}

Let's also give an example of a case where it makes sense to use more primes at once as $ p $-adic base
primes. Let's pick $ a = 3^2 \cdot 7^3 $, $ b = 17^3 $, and $ c = 2^6 \cdot 5^3 $, with quality
$ 1.09863\ldots $. The original method of Karsten and Taktikos with $ d = 3 $ would give $ b' = c' = 1 $,
$ \beta = 17 $, $ \gamma = 20 $, and one gets $ \dfrac{20}{17} $ as a `good' approximation to
$ 1^{1/3} $, which is a bit silly. It would give a real approximation gain of only
\[ \dfrac{\log\left(2^6 \cdot 5^3\right)}{\log\left(3^2 \cdot 7^3 \cdot 17 \cdot 2^2 \cdot 5\right)} =
    0.64824\ldots . \]
A second shot would be to look at the largest exponent, which occurs at the prime $ 2 $, and therefore work
$ 2 $-adically. We would get $ -\dfrac{17}{7} $ as a rational approximation $ 2 $-adically close to $ 3^{2/3}
$, and the $ 2 $-adic approximation gain then is
\[ \dfrac{\log\left(2^6 \cdot 5^3\right)}{\log\left(3^2 \cdot 7 \cdot 17 \cdot 2 \cdot 5^3\right)} =
    0.71910\ldots .\]
The $ p $-adic approach with, e.g., $ p = 7 $, does even better with an approximation gain of
$ 0.90123\ldots $.

But we can even do better, when we look $ p $-adically for the primes $ 2 $ and $ 5 $ simultaneously. As we
have third powers in $ a $ and $ b $ we look for good rational approximations of $ -3^{2/3} $. We take
approximation degrees $ 2^6 $ and $ 5^3 $. It's not hard to see that
$ -3^{2/3} \equiv 39 \pmod{2^6} $, $ -3^{2/3} \equiv 56 \pmod{5^3} $, and the Chinese Remainder Theorem then
gives $ -3^{2/3} \equiv 3431 \pmod{2^6 \cdot 5^3} $. So we compute the real continued fraction of
$ \dfrac{3431}{2^6 \cdot 5^3} $, being $ [0, 2, 3, 66, 1, 16] $, the approximation corresponding to the
biggish partial quotient of $ 66 $ is $ \dfrac{3}{7} $, and we find the expected answer by
$ 7 \cdot 3431 - 3 \cdot 8000 = 17 $ to be $ \dfrac{17}{7} $. For computing the improved approximation gain we
replace the powers $ 2^6, 5^3 $ by the primes $ 2, 5 $ only, and get
\[ \dfrac{\log\left(2^6 \cdot 5^3\right)}{\log\left(3^2 \cdot 7 \cdot 17 \cdot 2 \cdot 5\right)} =
    0.96855\ldots \]
which looks not too bad. See section \ref{sec:3.4} for a continuation of this example.

Clearly with the $ p $-adic case we can make any of the three of $ a, b, c $ small in some $ p $-adic sense,
and we can get three different $ p $-adic approximation gains. It makes sense to take the biggest of them as
\emph{the} $ p $-adic approximation gain. The exact definitions are written out in section \ref{sec:3.3}.

\subsubsection{Example 3 -- Reyssat}

Let's finally see what happens with Reyssat's $ abc $-triple $ 2 + 3^{10} \cdot 109 = 23^5 $, with quality
$ 1.62991\ldots $, where we found an approximation gain of $ 1.46283\ldots $. Computing with all four
occurring primes $ p = 2, 3, 23, 109 $ we computed the best $ p $-adic approximation gains, see Table
\ref{tab:Reyssat_padic}.

\begin{table}[ht] \centering
{\small $ \begin{array}{c|cc|c}
          p & \text{surd}                         & \text{approximation} & \text{approximation gain} \\ \hline
          2 & 109^{1/5}                           & \dfrac{23}{9}        & 1.46283\ldots             \\[2mm]
          3 & 2^{-1/5}                            & \dfrac{1}{23}        & 1.62991\ldots             \\
         23 & -\left(\dfrac{2}{109}\right)^{1/10} & \dfrac{3}{1}         & 1.62991\ldots             \\
        109 & 2^{-1/5}                            & \dfrac{1}{23}        & 0.80372\ldots             \\
    {3,109} & 2^{-1/5}                            & \dfrac{1}{23}        & 1.62991\ldots             \\
  \end{array} $}
\caption{$ p $-adic approximation and power gains for Reyssat's triple.}
\label{tab:Reyssat_padic}
\end{table}

For $ p = 2 $ the whole thing is almost trivial as there's an exponent of only $ 1 $, so we're computing
modulo $ 2 $ only, but this still gives the same $ 2 $-adic approximation gain
as in the real case. For $ p = 3 $ and $ p = 23 $ we reach even the best possible $ p $-adic
approximation gains, equal to the quality. For $ p = 109 $ the approximation gain is not very good, which was
to be expected as $ 109 $ has only exponent $ 1 $.

In general it makes sense to choose one of $ a, b, c $, and take it as a whole as a product of prime
powers that is multiple $ p $-adically small. In this case in the enhanced radical the contribution of
this number is set to be minimal, namely to only the product of its prime factors with multiplicity
$ 1 $, and the $ p $-adic approximation gain so is maximized.

\subsection{Definitions} \label{sec:3.3}

We now develop various definitions for $ p $-adic approximation gains.

\begin{definition} \label{def:p_gain_d}
Let $ (a,b,c) $ be an $ abc $-triple, and let $ d \geq 2 $ be any integer. Let
$ \left(a^{\ast}, b^{\ast}, c^{\ast}\right) $ be any permutation of $ (a,b,c) $.
\begin{enumerate}[topsep=0pt]
\item Let $ \beta $, $ b' $, $ \gamma $, $ c' $ be positive integers with (for given $ d $) $ \beta $,
      $ \gamma $ as large as possible such that $ b^{\ast} = \beta^d b' $, $ c^{\ast} = \gamma^d c' $. The
      \emph{multiple $ p $-adic enhanced radicals} $ \mrad_{a^{\ast},d}(a,b,c) $ are defined by
      \[ \mrad_{a^{\ast},d}(a,b,c) = \rad(a^{\ast}) \beta b' \gamma c'. \]
\item The \emph{$ d $th degree multiple $ p $-adic approximation gains} $ \mpag_{a^{\ast},d}(a,b,c) $ and the
      \emph{$ d $th degree multiple $ p $-adic power gains} $ \mpg_{a^{\ast},d}(a,b,c) $ are defined by
      \[ \mpag_{a^{\ast},d}(a,b,c) = \dfrac{\log c}{\log\mrad_{a^{\ast},d}(a,b,c)}, \quad
      \mpg_{a^{\ast},d}(a,b,c) = \dfrac{\log\mrad_{a^{\ast},d}(a,b,c)}{\log\rad(abc)}. \]
\end{enumerate}
\end{definition}

Note that the symbol $ p $ in the naming of those approximation gains is not explicitly present at all in the
definition's formul\ae. The multiple $ p $-adic approach uses all primes $ p $ in $ a^{\ast} $ to get the
most out of the $ p $-adic metrics. By taking $ \rad(a^{\ast}) $ instead of $ a^{\ast} $ itself in the
enhanced radical, we take care of the fact that $ a^{\ast} $ is multiple $ p $-adically small.

\begin{definition} \label{def:p_gain}
For an $ abc $-triple $ (a, b, c) $, the \emph{multiple $ p $-adic approximation gain} $ \mpag(a,b,c) $ and
the \emph{multiple $ p $-adic power gain} $ \mpg(a,b,c) $ are defined by
\[ \mpag(a,b,c) = \max_{d\geq2,a^{\ast}\in\{a,b,c\}} \mpag_{a^{\ast},d}(a,b,c), \quad
   \mpg(a,b,c) = \min_{d\geq2,a^{\ast}\in\{a,b,c\}} \mpg_{a^{\ast},d}(a,b,c) . \]
\end{definition}

To compute the multiple $ p $-adic approximation gain one even does not have to do any computation with
$ p $-adic numbers. It is totally straightforward to compute it directly from the multiple $ p $-adic
enhanced radicals as defined in Definition \ref{def:p_gain_d}. Only to understand the nature of the
$ p $-adic approximation and the underlying $ p $-adic approximation lattice / continued fraction, one has to
understand the $ p $-adics of the situation.

One could also be stricter and allow for the $ p $-adic approximation gain only single prime powers.

\begin{definition} \label{def:sp_gain_d}
Let $ (a,b,c) $ be an $ abc $-triple, and let $ d \geq 2 $ be any integer. Let
$ \left(a^{\ast}, b^{\ast}, c^{\ast}\right) $ be any permutation of $ (a,b,c) $. Let $ p $ be a prime with
$ p \mid a^{\ast} $, with exponent $ e_p(a) $.
\begin{enumerate}[topsep=0pt]
\item Let $ \beta $, $ b' $, $ \gamma $, $ c' $ be positive integers with (for given $ d $) $ \beta $,
      $ \gamma $ as large as possible such that $ b^{\ast} = \beta^d b' $, $ c^{\ast} = \gamma^d c' $. The
      \emph{enhanced radicals} $ \srad_{p,d}(a,b,c) $ are defined by
      \[ \srad_{p,d}(a,b,c) = \dfrac{a^{\ast}}{p^{e_p(a)-1}} \beta b' \gamma c'. \]
\item The \emph{$ d $th degree single $ p $-adic approximation gains} $ \sag_{p,d}(a,b,c) $ and the
      \emph{$ d $th degree single $ p $-adic power gains} $ \spg_{p,d}(a,b,c) $ are defined by
      \[ \sag_{p,d}(a,b,c) = \dfrac{\log c}{\log\srad_{p,d}(a,b,c)}, \quad
         \spg_{p,d}(a,b,c) = \dfrac{\log \srad_{p,d}(a,b,c)}{\log \rad(a,b,c)}. \]
\end{enumerate}
\end{definition}

\begin{definition} \label{def:sp_gain}
For an $ abc $-triple $ (a, b, c) $, the \emph{single $ p $-adic approximation gain} $ \sag(a,b,c) $ and
the \emph{single $ p $-adic power gain} $ \spg(a,b,c) $ are defined by
\[ \sag(a,b,c) = \max_{d\geq2, p|abc \textrm{\emph{ prime}}} \sag_{p,d}(a,b,c), \quad
   \spg(a,b,c) = \min_{d\geq2, p|abc \textrm{\emph{ prime}}} \spg_{p,d}(a,b,c) . \]
\end{definition}

\subsection{Example 2 revisited} \label{sec:3.4}
We continue the example from section \ref{sec:3.2.2}. We have $ 8 $ choices, namely real (indicated by
$ p = \infty $), the single $ p $-adic cases $ p = 2, 3, 5, 7, 17 $, and the $ p $-adic combinations
$ \{ 2, 5 \}, \{ 3, 7 \} $. The data are in Table \ref{tab:p-adic}, where always $ d = 3 $, and by $ r $ we
denote the enhanced radical.

\begin{table}[ht] \centering
{\small $ \begin{array}{c|cc|c}
        p & \text{surd} & \text{approximation} & \text{gain} \\ \hline
        \infty      & \left(\dfrac{1}{1}\right)^{1/3} & \dfrac{20}{17} & \rag_3 = 0.64824\ldots \\
                  2 & \left(\dfrac{9}{1}\right)^{1/3} & \dfrac{17}{7}  & \sag_{2,3} = 0.71910\ldots \\
                  3 & \left(\dfrac{1}{1}\right)^{1/3} & \dfrac{20}{17} & \sag_{3,3} = 0.70403\ldots \\
                  5 & \left(\dfrac{9}{1}\right)^{1/3} & \dfrac{17}{7}  & \sag_{5,3} = 0.70517\ldots \\
                  7 & \left(\dfrac{1}{1}\right)^{1/3} & \dfrac{20}{17} & \sag_{7,3} = 0.90123\ldots \\
                 17 & -\left(\dfrac{9}{1}\right)^{1/3} & \dfrac{20}{7} & \sag_{17,3} = \mpag_{b,3} =
                                                                       0.90123\ldots \\
            \{2,5\} & \left(\dfrac{9}{1}\right)^{1/3} & \dfrac{17}{7}  & \mpag_{c,3} = 0.96855\ldots \\
            \{3,7\} & \left(\dfrac{1}{1}\right)^{1/3} & \dfrac{20}{17} & \mpag_{a,3} = 1.01281\ldots \\
        \end{array} $}
\caption{Real and $ p $-adic approximation and power gains for
    $ (a,b,c) = (3^2 \cdot 7^3, 17^3, 2^6 \cdot 5^3) $.}
\label{tab:p-adic}
\end{table}

It turns out that $ \rag = 0.64824\ldots $, $ \mpag = 1.01281\ldots $, and $ \sag = 0.90123\ldots $.

\subsection{Computational results}

The various $ p $-adic approximation gains were computed for the $ 14482065 $ $ abc $-triples below
$ 10^{18} $, and for the known $ 9345651 $ $ abc $-triples between $ 10^{18} $ and $ 2^{63} $, from Bart de
Smit's tables \cite{dS}. The top 10 multiple $ p $-adic approximation gain $ abc $-triples are given in Table
\ref{tab:top10_mult}, the top 10 single ones in Table \ref{tab:top10_single}. Both the largest multiple and
single $ p $-adic approximation gain, $ 1.62991\ldots $, occurs, not unexpectedly, for Reyssat's
$ abc $-triple.

\begin{table}[ht] \centering
    {\small $ \begin{array}{ccc|cc}
            a & b & c & \mpag & \qu \\ \hline
            2 & 3^{10} \cdot 109 & 23^5 & 1.62991\ldots & 1.62991\ldots \\
            1 & 2 \cdot 3^7 & 5^4 \cdot 7 & 1.56788\ldots & 1.56788\ldots \\
            37 & 2^{15} & 3^8 \cdot 5 & 1.48291\ldots & 1.48291\ldots \\
            1 & 3^{16} \cdot 7 & 2^3 \cdot 11 \cdot 23 \cdot 53^3 & 1.47444\ldots & 1.47444\ldots \\
            2^7 \cdot 5^2 & 7^6 \cdot 41 & 13^6 & 1.46192\ldots & 1.46192\ldots \\
            1 & 2^5 \cdot 3 \cdot 5^2 & 7^4 & 1.45567\ldots & 1.45567\ldots \\
            1 & 3 \cdot 5^5 \cdot 47^2 & 2^{18} \cdot 79 & 1.44965\ldots & 1.44965\ldots \\
            89 & 7 \cdot 11^8 & 2^{20} \cdot 3^3 \cdot 53 & 1.44774\ldots & 1.44774\ldots \\
            1 & 19 \cdot 509^3 & 2^{19} \cdot 3^4 \cdot 59 & 1.43835\ldots & 1.43835\ldots \\
            11 & 7^3 \cdot 167^2 & 2 \cdot 3^{14} & 1.42832\ldots & 1.42832\ldots \\
        \end{array} $}
    \caption{Top 10 multiple $ p $-adic approximation gains known.}
    \label{tab:top10_mult}
\end{table}

\begin{table}[ht] \centering
    {\small $ \begin{array}{ccc|cc}
            a & b & c & \sag & \qu \\ \hline
            2 & 3^{10} \cdot 109 & 23^5 & 1.62991\ldots & 1.62991\ldots \\
            1 & 2 \cdot 3^7 & 5^4 \cdot 7 & 1.56788\ldots & 1.56788\ldots \\
            37 & 2^{15} & 3^8 \cdot 5 & 1.48291\ldots & 1.48291\ldots \\
            1 & 3^{16} \cdot 7 & 2^3 \cdot 11 \cdot 23 \cdot 53^3 & 1.47444\ldots & 1.47444\ldots \\
            3 & 5^3 & 2^7 & 1.42656\ldots & 1.42656\ldots \\
            2^6 \cdot 5 \cdot 137 & 3^{14} & 13^6 & 1.41369\ldots & 1.41369\ldots \\
            5 & 3^{11} & 2^{10} \cdot 173 & 1.41268\ldots & 1.41268\ldots \\
            7^3 & 3^{10} & 2^{11} \cdot 29 & 1.40956\ldots & 1.54707\ldots \\
            7 \cdot 13 & 2 \cdot 3^{11} \cdot 5 & 11^6 & 1.39548\ldots & 1.39548\ldots \\
            1 & 11^4 \cdot 47 & 2^{15} \cdot 3 \cdot 7 & 1.34609\ldots & 1.34609\ldots \\
        \end{array} $}
    \caption{Top 10 single $ p $-adic approximation gains known.}
    \label{tab:top10_single}
\end{table}

The largest multiple $ p $-adic power gain observed is $ \mpg(a,b,c) = 1.69801\ldots $ at
$ (a,b,c) = (29^3 \cdot 79^5, 2^{31} \cdot 17^3 \cdot 31^2, 3^4 \cdot 5^9 \cdot 7 \cdot 3037^2) $, with
$ \rag(a,b,c) = 0.50951\ldots $, $ \qu(a,b,c) = 1.34683\ldots $.

The largest single $ p $-adic power gain observed is $ \spg(a,b,c) = 2.15027\ldots $ at
$ (a,b,c) = (7^{10} \cdot 19^6 \cdot 37, 2^8 \cdot 17^5 \cdot 271^2 \cdot 349^2, 3^{16} \cdot 5^9 \cdot
 211^2) $, with $ \rag(a,b,c) = 0.49929\ldots $, $ \qu(a,b,c) = 1.35574\ldots $.

Note that in Tables \ref{tab:top10_mult} and \ref{tab:top10_single} of the 20 cases 19 have approximation
gain equal to the quality.

Much more data are available on the website \cite{dW26} for this paper.

\section{Combining the real and $ p $-adic approximation gains}
The multiple $ p $-adic approximation gain is always at least as large as the real approximation gain, even
in the case of $ a $ being extremely small in the real sense. This is because the contribution of $ a $ to
the enhanced radical is always the best in the $ p $-adic case. So in this case the `combined' real and
multiple $ p $-adic approximation gain will just be the multiple $ p $-adic approximation gain. We think this
observation is important enough to call it a theorem.

\begin{theorem} \label{thm:gain}
    For all $ abc $-triples
    \[ \mpag(a,b,c) \geq \rag(a,b,c) . \]
\end{theorem}

\textbf{Proof.}
Definition \ref{def:p_gain} shows that $ \mpag(a,b,c) \geq \mpag_{a^{\ast},d}(a,b,c) $ for all $ d $ and for
those permutations $ (a^{\ast},b^{\ast},c^{\ast}) $ of $ (a,b,c) $ satisfying $ a^{\ast} = a $. For the
enhanced radicals we have, by Definitions \ref{def:gain_d} and \ref{def:p_gain_d},
$\mrad_{a^{\ast},d}(a,b,c) = \rad(a^{\ast}) \beta b' \gamma c' =
    \rad(a) \beta b' \gamma c' \leq a \beta b' \gamma c' = \rrad_d(a,b,c) $, so
$ \mpag_{a^{\ast},d}(a,b,c) = \dfrac{\log c}{\mrad_{a^{\ast},d}(a,b,c)} \geq \dfrac{\log c}{\rrad_d(a,b,c)} $
for all $ d $. The result follows by Definition \ref{def:gain}. \hfill $ \square $ \vspace*{4mm}

For the single prime case usually $ \sag(a,b,c) \geq \rag(a,b,c) $, but it might happen that
$ \sag(a,b,c) < \rag(a,b,c) $, an example is
$ (a,b,c) = (1, 2^{19} \cdot 5^3 \cdot 4909^3, 11^5 \cdot 31^4 \cdot 43 \cdot 601 \cdot 2017) $ with
$ \sag(a,b,c) = 0.90325\ldots $, $ \rag(a,b,c) = 0.92954\ldots $, $ \mpag(a,b,c) = 1.18190\ldots $,
$ \qu(a,b,c) = 1.26428\ldots $. So in such cases it does make sense to introduce a combined real and single
$ p $-adic approximation gain as follows.

\begin{definition} \label{def:tot_gain}
For an $ abc $-triple $ (a, b, c) $, the \emph{combined approximation gain} $ \cag(a,b,c) $ and the
\emph{combined power gain} $ \cpg(a,b,c) $ are defined by
\[ \cag(a,b,c) = \max\{\rag(a,b,c),\sag(a,b,c)\}, \quad \cpg(a,b,c) = \min\{\rpg(a,b,c),\spg(a,b,c)\} . \]
\end{definition}

\section{Computational results}

We used the data on $ abc $-triples on Bart de Smit's webpage \cite{dS}, containing all 14\,482\,065
$ abc $-triples with $ c < 10^{18} $, all 9\,345\,651 $ abc $-triples with $ 10^{18} \leq c < 2^{63} $, and
371 special $ abc $-triples with $ c \geq 2^{63} $, namely those of high quality, high merit, or unbeaten. Of
all those triples we computed the real and $ p $-adic approximation and power gains. Figure \ref{fig:distr}
shows the distributions of the approximation and power gains in bins of size $ 0.01 $.

\begin{figure}[h]
    \centering
    \includegraphics[width=0.68\textwidth]{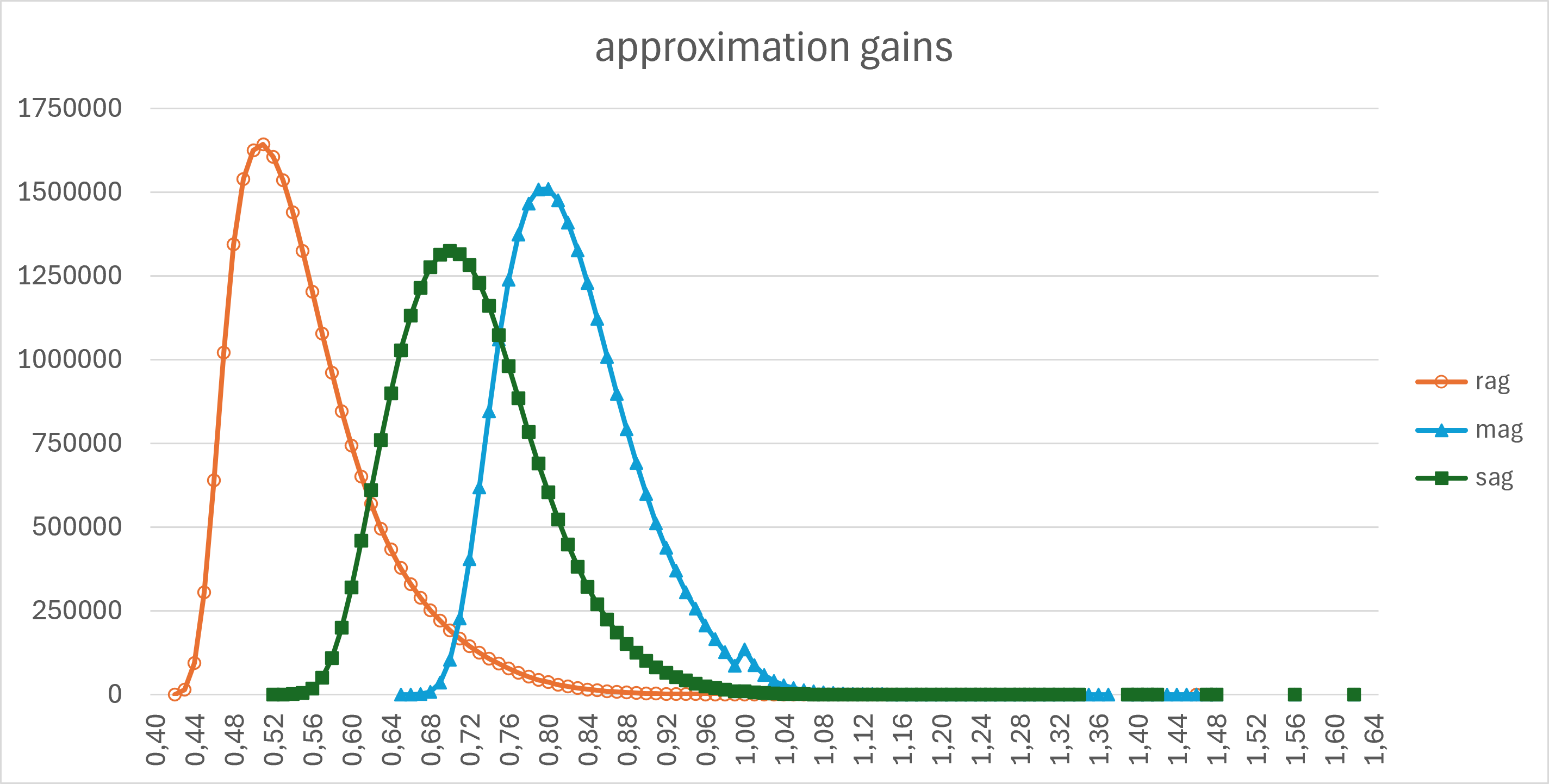} \\[2mm]
    \includegraphics[width=0.68\textwidth]{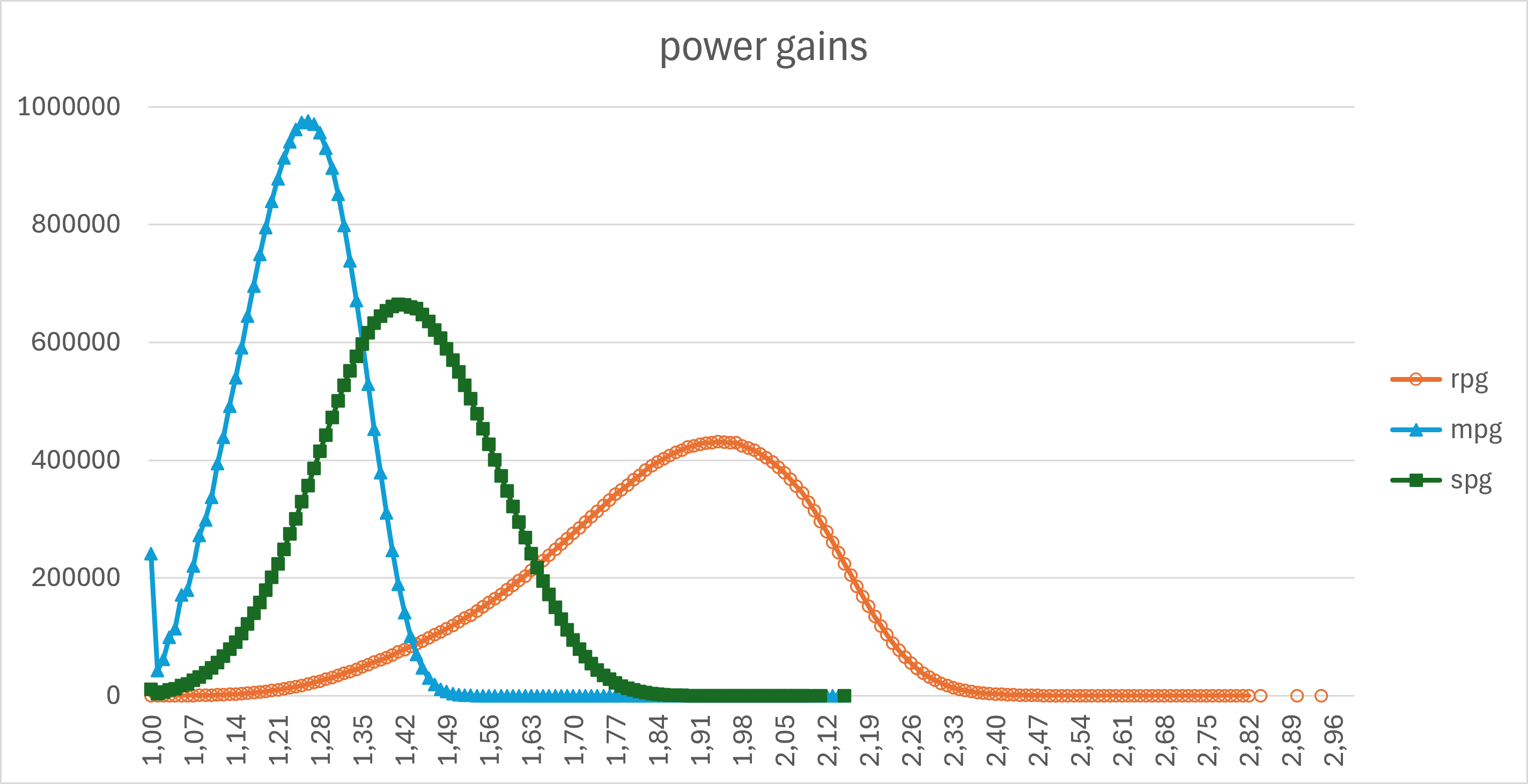}
    \caption{Distributions of the approximation and power gains.}
    \label{fig:distr}
\end{figure}

On the webpage \cite{dW26} for this paper we present tables with $ abc $-triples of high gain, according to
the selection criteria in Table \ref{tab:selcrit}.

\begin{table}[ht] \centering
{\small
\[ \begin{array}{c|cccccc}
    c & \rag > & \text{or } \rpg > & \text{or } \mpag > & \text{or } \mpg > & \text{or } \sag > & \text{or } \spg > \\ \hline
    < 10^{18}              & 1.10 & 2.66 & 1.37 & 1.57 & 1.31 & 1.90 \\
    \geq 10^{18}, < 2^{63} & 0.90 & 2.45 & 1.09 & 1.51 & 1.02 & 1.86 \\
    \geq 2^{63}            & \multicolumn{6}{l}{\text{no additional criteria}}
    \end{array} \]
}
    \caption{Selection criteria for published tables.}
\label{tab:selcrit}
\end{table}

This resulted in $ 1133 $ $ abc $-triples with $ c < 10^{18} $, $ 13302 $ $ abc $-triples with
$ 10^{18} \geq c < 2^{63} $ and $ 371 $ $ abc $-triples with $ c \geq 2^{63} $. Figure \ref{fig:scatplot}
shows some scatterplots of those data. The structures one might observe in there are probably mostly
caused by our somewhat randomly chosen selection criteria.

\begin{figure}[h]
    \centering
    \includegraphics[width=0.32\textwidth]{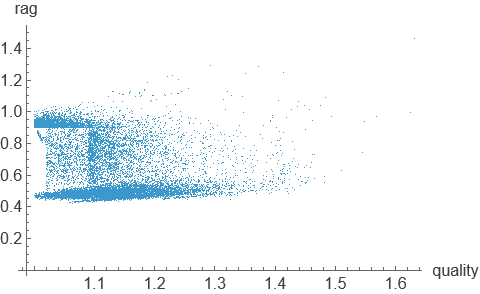}
    \includegraphics[width=0.32\textwidth]{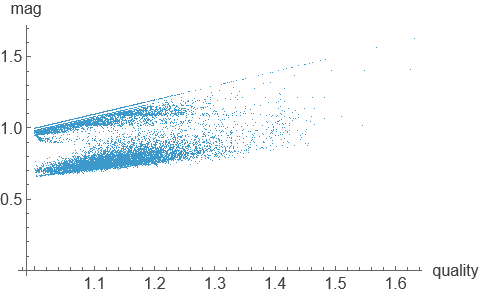}
    \includegraphics[width=0.32\textwidth]{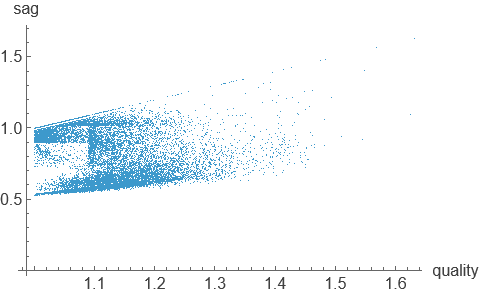} \\
    \includegraphics[width=0.32\textwidth]{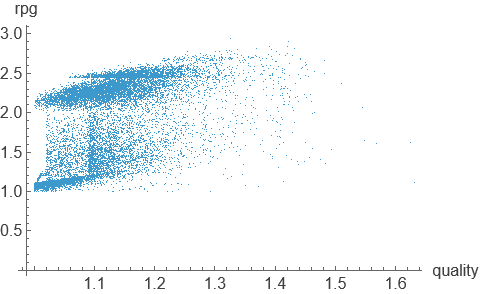}
    \includegraphics[width=0.32\textwidth]{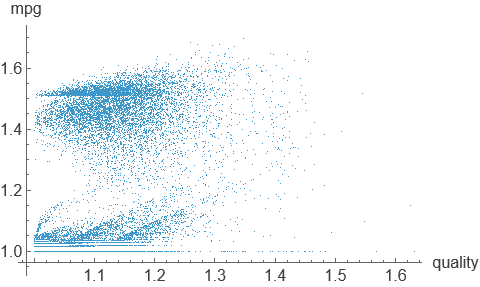}
    \includegraphics[width=0.32\textwidth]{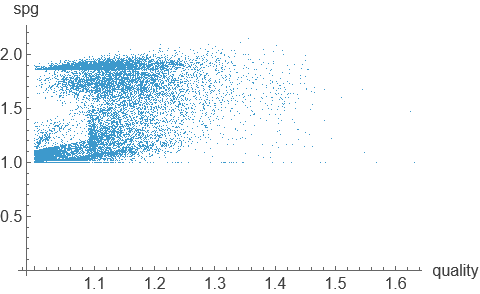} \\
    \includegraphics[width=0.32\textwidth]{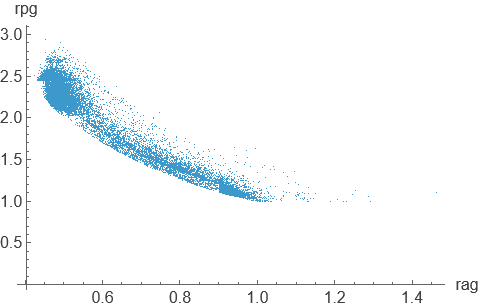}
    \includegraphics[width=0.32\textwidth]{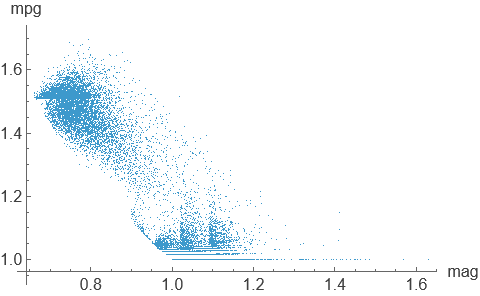}
    \includegraphics[width=0.32\textwidth]{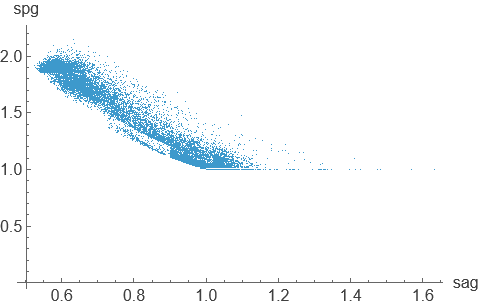} \\
    \includegraphics[width=0.32\textwidth]{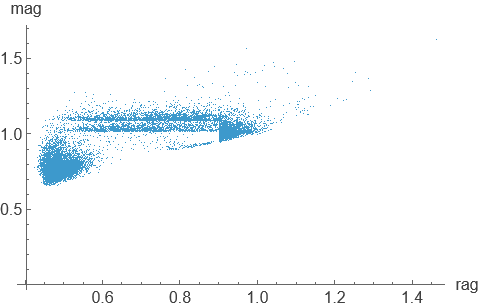}
    \includegraphics[width=0.32\textwidth]{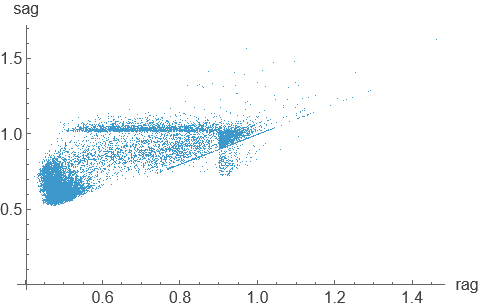}
    \includegraphics[width=0.32\textwidth]{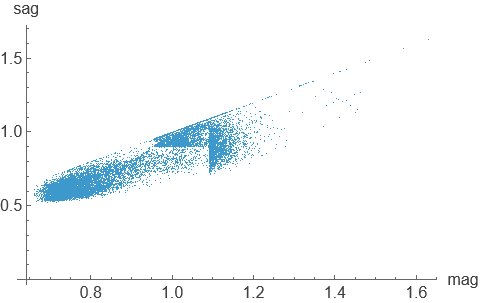} \\
    \caption{Scatterplots of observed quality, approximation and power gain data.}
    \label{fig:scatplot}
\end{figure}

\subsection*{Acknowledgements}
Thanks are due to Karsten M\"uller for his encouragement.

\bibliographystyle{plainurl}

\end{document}